\documentclass[12pt]{amsart}

\usepackage{amsmath}
\usepackage{amssymb}
\usepackage{enumerate}

\renewcommand{\pmod}[1]{\,(\textup{mod}\,#1)}
\numberwithin{equation}{section}
 \theoremstyle{plain}
\newtheorem{theorem}{Theorem}[section]

\newtheorem{lemma}[theorem]{Lemma}

\newcommand{\wh}{\widehat{\sigma}}
\newcommand{\wt}{\widetilde{\sigma}}
\newcommand{\mP}{\mathcal{P}}
\newcommand{\mQ}{\mathcal{Q}}
\newcommand{\mE}{\mathcal{E}}

\newenvironment{pf}
   {\par\noindent{\it Proof of}\hskip 0.5em\ignorespaces}
   {\hfill $\Box$\par\medskip}

\begin{document}

\title{Convolution sums of some functions on divisors}
\author{Heekyoung Hahn}
\address{Department of Mathematics, University of Rochester, Rochester, NY 14627 USA}
\email{hahn@math.rochester.edu}

\begin{abstract}
One of the main goals in this paper is to establish convolution sums of functions for the divisor sums $\wt_s(n)=\sum_{d|n}(-1)^{d-1}d^s$ and $\wh_s(n)=\sum_{d|n}(-1)^{\frac{n}{d}-1}d^s$, for certain $s$, which were first defined by Glaisher. We first introduce three functions $\mP(q)$, $\mE(q)$, and $\mQ(q)$ related to $\wt(n)$, $\wh(n)$, and $\wt_3(n)$, respectively, and then we evaluate them in terms of two parameters $x$ and $z$ in Ramanujan's theory of elliptic functions. Using these formulas, we derive some identities from which we can deduce convolution sum identities. We discuss some formulae for determining $r_s(n)$ and $\delta_s(n)$, $s=4,$ $8$, in terms of $\wt(n)$, $\wh(n)$, and $\wt_3(n)$, where $r_s(n)$ denotes the number of representations of $n$ as a sum of $s$ squares and $\delta_s(n)$ denotes the number of representations of $n$ as a sum of $s$ triangular numbers. Finally, we find some partition congruences by using the notion of colored partitions.
\end{abstract}

\maketitle

%%%%%%%%%%%%%%%%%%%%%%%%%%%%%%%%%%%%%%%%%%%%%%%%%%%%%%%%%%%%%%%%%%%%%%%%%%
\section{Introduction}

In his famous paper \cite{Ramanujan-arith}, \cite[pp. 136--162]{Ramanujan-collect}, Ramanujan introduced the three Eisenstein series $P(q)$, $Q(q)$ and $R(q)$ defined for $|q|<1$ by
\begin{align}
P(q)&:=1-24\sum_{n=1}^{\infty} \sigma(n)q^n,\label{P} \\
Q(q)&:=1+240\sum_{n=1}^{\infty} \sigma_3(n)q^n,\label{Q}\\
R(q)&:=1-504\sum_{n=1}^{\infty} \sigma_5(n)q^n, \label{R}
\end{align} where for $s, n \in \mathbb{N}$, $$\sigma_s(n)=\sum_{d|n}d^s.$$ 
As usual, we set $\sigma_1(n)=\sigma(n)$ and $\sigma_s(n)=0$ if $n
\notin \mathbb{N}$. Ramanujan also proved that \eqref{P}--\eqref{R} satisfy the differential
equations~\cite[(30)]{Ramanujan-arith}, \cite[p. 142]{Ramanujan-collect}
\begin{align}
q\frac{dP(q)}{dq}&=\frac{P^2(q)-Q(q)}{12}, \label{dP}\\
q\frac{dQ(q)}{dq}&=\frac{P(q)Q(q)-R(q)}{3}, \label{dQ}\\
q\frac{dR(q)}{dq}&=\frac{P(q)R(q)-Q^2(q)}{2}. \label{dR}
\end{align} After rewriting \eqref{dP} as
\begin{equation}
 P^2(q)=Q(q)+12q\frac{dP(q)}{dq},
\end{equation}
and equating the coefficients of $q^n$ on both sides, we obtain the arithmetic identity
\begin{equation}\label{sigma1}
12\sum_{m<n}\sigma(m)\sigma(n-m)=5\sigma_3(n)-(6n-1)\sigma(n).
\end{equation}
Likewise, from \eqref{dQ}, we obtain
\begin{equation}\label{sigma2} 
240\sum_{m<n}\sigma(m)\sigma_3(n-m)=21\sigma_5(n)-(30n-10)\sigma_3(n)-\sigma(n).
\end{equation}
Ramanujan recorded nine identities of the type \eqref{sigma1} and \eqref{sigma2} in his notebooks. The history of the convolution sums involving the divisor function $\sigma_s(n)$ goes back to Glaisher~\cite{Glaisher-square}, \cite{Glaisher-sum},
\cite{Glaisher-express}. A most comprehensive treatment of these identities is given in the paper \cite{Huard}. In their paper \cite{Huard}, Huard, Ou, Spearman and Williams prove many such formulae in an elementary manner by using their generalization of Liouville's classical formula given in~\cite{Liouville}. Recently, Cheng and Williams~\cite{ChengWill} found further convolution sums of the type $$\sum_{m<n}\sigma(4m-3)\sigma(4n-(4m-3))=4\sigma_3(n)-4\sigma_3(n/2).$$

Now define two functions on which we focus in this paper by, for $s, n \in \mathbb{N}$,
\begin{align}
\wt_s(n)&=\sum_{d|n}(-1)^{d-1}d^s, \label{til}\\
\wh_s(n)&=\sum_{d|n}(-1)^{\frac{n}{d}-1}d^s, \label{hat}
\end{align}
where we set $\wt_1(n)=\wt(n),~\wh_1(n)=\wh(n)$ and $\wt_s(n)=\wh_s(n)=0$ if $n \notin \mathbb{N}$. The origin of these functions goes back to Glaisher. In his paper~\cite{Glaisher-sum}, Glaisher defined seven quantities which depend on the divisors of $n$, including \eqref{til} and \eqref{hat}, and studied the relations among them. He also found expressions for all seven functions in terms of $\sigma_s(n)$. For instance, the functions $\wt_s(n)$ and $\wh_s(n)$ have the formulae~\cite{Glaisher-sum}
\begin{align}
\wt_s(n)&=\sigma_s(n)-2^{s+1}\sigma_s(n/2), \label{wt-s}\\
\wh_s(n)&=\sigma_s(n)-2\sigma_s(n/2). \label{wh-s}
\end{align} From the relations \eqref{wt-s} and \eqref{wh-s}, it is clear that, for all $n \geq 0$,
\begin{equation}
\wt_s(2n+1)=\sigma_s(2n+1)=\wh_s(2n+1).\label{odd}
\end{equation}

One of our goals in the present paper is to establish convolution sums involving $\wt_s$ and $\wh_s$ for certain $s$. So we need to define three functions related to \eqref{til} and \eqref{hat} by, for $|q|<1,$
\begin{align}
\mP(q)&:=1+8\sum_{n=1}^{\infty} \wt(n)q^n, \label{mp}\\
\mE(q)&:=1+24\sum_{n=1}^{\infty} \wh(n)q^n, \label{me}\\
\mQ(q)&:=1-16\sum_{n=1}^{\infty} \wt_3(n)q^n. \label{mq}
\end{align} 
Analogous to \eqref{dP}--\eqref{dR}, our three functions \eqref{mp}--\eqref{mq} satisfy the differential equations \cite{Hahn}, \cite{Ramamani-thesis}, \cite{Ramamani-paper}
\begin{align}
q\frac{d\mP(q)}{dq}&=\frac{\mP^2(q)-\mQ(q)}{4}, \label{dmp}\\
q\frac{d\mE(q)}{dq}&=\frac{\mE(q)\mP(q)-\mQ(q)}{2}, \label{dme}\\
q\frac{d\mQ(q)}{dq}&=\mP(q)\mQ(q)-\mE(q)\mQ(q). \label{dmq}
\end{align}
If we define the related series analogues to \cite{ChengWill}
\begin{align}
\mP_{r,2}(q)&=\sum_{n=0}^{\infty}\wt(2n+r)q^{2n+r}, \quad r=0,1,\label{sp}\\
\mE_{r,2}(q)&=\sum_{n=0}^{\infty}\wh(2n+r)q^{2n+r}, \quad r=0,1,\label{se}\\
\mQ_{r,2}(q)&=\sum_{n=0}^{\infty}\wt_3(2n+r)q^{2n+r},~\quad~r=0,1,\label{sq}
\end{align}
then we find many identities involving the series $\mP_{r,2}(q)$, $\mE_{r,2}(q)$, $\mQ_{r,2}(q)$, and the functions $\mP(q)$, $\mE(q)$, and $\mQ(q)$.

In Section 2, we evaluate \eqref{mp}, \eqref{me}, \eqref{mq}, \eqref{sp}, \eqref{se}, and \eqref{sq} in terms of two parameters $x$ and $z$ in Ramanujan's theory of elliptic functions. Using these formulas, we derive some identities involving Ramanujan's theta functions. In Section 3, we find representations for certain infinite series related to $\mP(q)$, $\mE(q)$, and $\mQ(q)$. In Section 4, using the evaluations we obtained in Section 2, we derive convolution sums of \eqref{til} and \eqref{hat}. In Section 5, we discuss some formulae for determining $r_s(n)$ and $\delta_s(n)$ in terms of $\wt_s(n)$ and $\wh_s(n)$, where $r_s(n)$ denotes the number of representations of $n$ as a sum of $s$ squares and $\delta_s(n)$ denotes the number of representations of $n$ as a sum of $s$ triangular numbers. Finally, we find some partition congruences connected with $\wt_s(n)$ and $\wh_s(n)$ by using the notion of colored partitions.

%%%%%%%%%%%%%%%%%%%%%%%%%%%%%%%%%%%%%%%%%%%%%%%%%%%%%%%%%%%%%%%%%%%%%%%%%%%%%
\section{Evaluations and identities involving Ramanujan's theta functions}

To derive the desired identities, we need to use evaluations of theta functions \cite[pp. 122--138]{BCB3} to determine the quantities $\mP(q^r)$, $\mE(q^r)$, $\mQ(q^r)$, $\mP(-q)$, $\mE(-q)$, and $\mQ(-q)$, $r=1, 2$.

If
\begin{equation}
y=\pi \frac{{}_2F_1 \left( \frac{1}{2},\frac{1}{2};1;1-x
\right)}{{}_2F_1 \left( \frac{1}{2},\frac{1}{2};1;x
\right)},~|x|<1, \label{y}
\end{equation}
where ${}_2F_1$ denotes the Gaussian hypergeometric function, the evaluations are given in terms of, in Ramanujan's notation,
\begin{equation}
z:={}_2F_1 \left( \frac{1}{2},\frac{1}{2};1;x \right) \label{z}
\end{equation}
and $x$. The derivative $y'$ is given by
\begin{equation}
\frac{dy}{dx}=-\frac{1}{x(1-x)z^2}; \label{dy}
\end{equation}
see, for example, Berndt's book \cite[p. 87]{BCB2}. The function
$z:={}_2F_1 \left(\frac{1}{2},\frac{1}{2};1;x \right)$ satisfies
the differential equation \cite[p. 120]{BCB3}
\begin{equation}\label{alt-dz}
\frac{d^2z}{dz^2}=\frac{z}{4x(1-x)}-\frac{(1-2x)}{x(1-x)}\frac{dz}{dx}.
\end{equation}

 From now on, we will denote $$q:=e^{-y}.$$
Ramanujan's theta functions $\varphi(q)$, $\psi(q)$, and $f(-q)$ \cite[Entry
22, p. 36]{BCB3} are defined, for $|q|<1$, by
{\allowdisplaybreaks\begin{align}
\varphi(q)&:=\sum_{n=-\infty}^{\infty} q^{n^2}=\frac{(-q;q^2)_{\infty}(q^2;q^2)_{\infty}}{(q;q^2)_{\infty}(-q^2;q^2)_{\infty}},\label{phi}\\
\psi(q)&:=\sum_{n=0}^{\infty}q^{n(n+1)/2}=\frac{(q^2;q^2)_{\infty}}{(q;q^2)_{\infty}},\label{psi}\\
f(-q)&:=\sum_{n=-\infty}^{\infty}(-1)^nq^{n(3n+1)/2}=(q;q)_{\infty},\label{f}
\end{align}}
where, as usual, for any complex number $a$, we write $$(a;q)_{\infty}:=\prod_{n=0}^{\infty}(1-aq^n).$$ Here, the product representations arise from the Jacobi triple product identity \cite[Entry 19, p. 35]{BCB3}. In the
following lemma, we list the evaluations of the theta
functions in terms of $x$ and $z$ \cite[Entries 10--12,
pp. 122--124]{BCB3}, which we will employ in a majority of our proofs.
\begin{lemma}\label{BCB}
If $y$ and $z$ are defined by \eqref{y}, \eqref{z},
respectively, and $\psi(q)$, $\varphi(q)$, and $f(-q)$ are defined
by \eqref{phi}, \eqref{psi}, and \eqref{f}, respectively, then
{\allowdisplaybreaks\begin{align}
\varphi(q)&=\sqrt{z},\label{phixz}\\
\varphi(-q)&=(1-x)^{1/4}\sqrt{z},\label{phixz-}\\
q^{1/8}\psi(q)&=2^{-1/2}x^{1/8}\sqrt{z},\label{psixz}\\
q^{1/4}\psi(q^2)&=2^{-1}x^{1/4}\sqrt{z},\label{psixz2}\\
q^{1/24}f(-q)&=2^{-1/6}(1-x)^{1/6}x^{1/24}\sqrt{z}.\label{fxz-}
\end{align}}
\end{lemma}
Using these evaluations, we obtain formulas for $\mP(q)$, $\mE(q)$, and $\mQ(q)$.
\begin{theorem}
If $y$ and $z$ are defined as in \eqref{y} and \eqref{z},
respectively, and $q:=e^{-y}$, then
{\allowdisplaybreaks\begin{align}
\mP(q)&=z^2(1-x)+4x(1-x)z\frac{dz}{dx}, \label{mpxz}\\
\mE(q)&=z^2(1+x), \label{mexz}\\
\mQ(q)&=z^4(1-x)^2. \label{mqxz}
\end{align}}
\end{theorem}

\pf \eqref{mpxz}. In the derivation below, we find that, by using \eqref{psixz}, 
{\allowdisplaybreaks\begin{align*}
\mP(q)&=1+8\sum_{n=1}^{\infty}\frac{(-1)^{n-1}n}{e^{ny}-1}\\
&=1-8\frac{d}{dy}\sum_{n=1}^{\infty}(-1)^n\text{Log}(1-e^{-ny})\\
&=1-8\frac{d}{dy}\text{Log}\prod_{n=1}^{\infty}\frac{1-e^{-2ny}}{1-e^{-(2n-1)y}}\\
&=-8\frac{d}{dy}\text{Log}\{e^{-y/8}\psi(e^{-y})\},
\end{align*}} where we use the infinite product representation of $\psi(e^{-y})$
in \eqref{psi}. If we employ \eqref{psixz} and \eqref{dy}, then we find that 
{\allowdisplaybreaks\begin{align*}
\mP(q)&=8x(1-x)z^2\frac{d}{dx}\text{Log}\{2^{-1/2}\sqrt{z}x^{1/8}\}\\
&=z^2(1-x)+4x(1-x)z\frac{dz}{dx}.
\end{align*}}

\pf \eqref{mexz}. In the derivation below, we employ \eqref{fxz-} and \eqref{phixz-} to find that
{\allowdisplaybreaks\begin{align*}
\mE(q)&=1+24\sum_{n=1}^{\infty}\frac{n}{e^{ny}+1}\\
&=1-24\frac{d}{dy}\sum_{n=1}^{\infty}\text{Log}(1+e^{-ny})\\
&=-24\frac{d}{dy}\text{Log}\left\{e^{-y/24}\frac{f(-e^{-y})}{\varphi(-e^{-y})}\right\}.
\end{align*}} Again using the evaluations \eqref{phixz-} and \eqref{fxz-} and applying \eqref{dy}, we find that
{\allowdisplaybreaks\begin{align*}
\mE(q)&=24x(1-x)z^2\frac{d}{dx}\text{Log} \{ 2^{-1/6}(1-x)^{-1/12}x^{1/24} \}\\
&=z^2(1+x),
\end{align*}} which completes our proof.
\medskip
\pf \eqref{mqxz}. From \eqref{dmp}, we have
$$q\frac{d\mP(q)}{dq}=\frac{\mP^2(q)-\mQ(q)}{4}.$$ Thus, by the chain
rule, we deduce that
$$\frac{d\mP(e^{-y})}{dy}=\frac{\mQ(e^{-y})-\mP^2(e^{-y})}{4}.$$
Moreover, by \eqref{dy}, we derive that
$$\frac{d\mP(e^{-y})}{dx}=-\frac{1}{x(1-x)z^2}\frac{d\mP(e^{-y})}{dy}.$$
Hence, \begin{equation}\label{i-mqxz}
-x(1-x)z^2\frac{d\mP(e^{-y})}{dx}=\frac{\mQ(e^{-y})-\mP^2(e^{-y})}{4}.
\end{equation}
Thus we see that we can determine $\mQ(e^{-y})$ from \eqref{mpxz}
and \eqref{i-mqxz}. Using \eqref{mpxz} and the hypergeometric
differential equation \eqref{alt-dz}, we find, upon direct calculation, that
\begin{equation}\label{ii-mqxz}
\frac{d\mP(e^{-y})}{dx}=2(1-x)z\frac{dz}{dx}+4x(1-x)\left(\frac{dz}{dx}\right)^2.
\end{equation} Thus from \eqref{mpxz}, \eqref{i-mqxz}, and \eqref{ii-mqxz}, we see that 
{\allowdisplaybreaks\begin{align*}
\mQ(q)=\mQ(e^{-y})=&\left \{(1-x)z^2+4x(1-x)z\frac{dz}{dx} \right \}^2 \\
&-4x(1-x)z^2\left\{2(1-x)z\frac{dz}{dx}+4x(1-x)\left(\frac{dz}{dx}\right)^2
\right \}.
\end{align*}} Upon simplifying, we reach the desired conclusion.

\medskip
Before proceeding further, we briefly mention the procedure \cite[p. 125]{BCB3}, called {\it duplication}, in the theory of elliptic functions. If
\begin{equation}\label{old}
\Omega (x,e^{-y},z)=0,
\end{equation} and $x'$, $y'$, and $z'$ is another set of parameters such that
$$\Omega (x',e^{-y'},z')=0$$and $$ x=\frac{4\sqrt{x'}}{(1+\sqrt{x'})^2},$$ then we can deduce the ``new'' formula
\begin{equation}\label{new}
\Omega \left( \left( \frac{1-\sqrt{1-x}}{1+\sqrt{1-x}} \right)^2,
e^{-2y}, \frac{1}{2}z(1+\sqrt{1-x}) \right)=0,
\end{equation} from the ``old'' formula~\eqref{old}.
This process is called {\it obtaining a formula by duplication}.
We will use this procedure in many proofs.

\medskip
Applying the process of duplication to \eqref{mpxz}, \eqref{mexz}, and \eqref{mqxz}, 
we obtain 
{\allowdisplaybreaks\begin{align}
\mP(q^2)&=z^2(1-x)+2x(1-x)z\frac{dz}{dx},\label{mpxz2}\\
\mE(q^2)&=z^2(1-\frac{1}{2}x), \label{mexz2}\\
\mQ(q^2)&=z^4(1-x). \label{mqxz2}
\end{align}}

Berndt \cite[p. 126]{BCB3} has also described the process of obtaining a new formula from \eqref{old} by changing the sign of $q$. If \eqref{old} holds then the formula
\begin{equation}\label{new-}
\Omega \left(\frac{x}{x-1}, -q, z\sqrt{1-x}\right)=0
\end{equation} also holds. This result is attributed to Jacobi by Berndt \cite[p. 126]{BCB3}.

Applying Jacobi's change of sign procedure to \eqref{mpxz},
\eqref{mexz} and \eqref{mqxz}, we deduce that 
{\allowdisplaybreaks\begin{align}
\mP(-q)&=z^2(1-2x)+4x(1-x)z\frac{dz}{dx}, \label{mpxz-}\\
\mE(-q)&=z^2(1-2x), \label{mexz-}\\
\mQ(-q)&=z^4. \label{mqxz-}
\end{align}}

Simple calculations analogus to \cite{ChengWill} show that
{\allowdisplaybreaks\begin{align}
\mP_{0,2}(q)=&\frac{1}{16}\big(-2+\mP(q)+\mP(-q)\big),\label{sp0}\\
\mP_{1,2}(q)=&\frac{1}{16}\big(\mP(q)-\mP(-q)\big),\label{sp1}\\
\mE_{0,2}(q)=&\frac{1}{48}\big(-2+\mE(q)+\mE(-q)\big),\label{se0}\\
\mE_{1,2}(q)=&\frac{1}{48}\big(\mE(q)-\mE(-q)\big),\label{se1}\\
\mQ_{0,2}(q)=&\frac{1}{32}\big(-2-\mQ(q)-\mQ(-q)\big),\label{sq0}\\
\mQ_{1,2}(q)=&\frac{1}{32}\big(-\mQ(q)+\mQ(-q)\big).\label{sq1}
\end{align}}
Using \eqref{mpxz}--\eqref{mqxz} and \eqref{mpxz-}--\eqref{mqxz-},
we obtain the evaluations of the series $\mP_{r,2}(q)$,
$\mE_{r,2}(q)$ and $\mQ_{r,2}(q)$ as follows:
\begin{theorem}
\begin{align}
\mP_{0,2}(q)=&\frac{1}{16}\big(-2+(2-3x)z^2+8x(1-x)z\frac{dz}{dx}\big),\label{sp0xz}\\
\mP_{1,2}(q)=&\frac{1}{16}xz^2,\label{sp1xz}\\
\mE_{0,2}(q)=&\frac{1}{48}\big(-2+(2-x)z^2\big),\label{se0xz}\\
\mE_{1,2}(q)=&\frac{1}{16}xz^2,\label{se1xz}\\
\mQ_{0,2}(q)=&\frac{1}{32}\big(2-(2-2x+x^2)z^4\big),\label{sq0xz}\\
\mQ_{1,2}(q)=&\frac{1}{32}x(2-x)z^4.\label{sq1xz}
\end{align}
\end{theorem}

We note a few results which are used in the next section. Using \eqref{dy} and $q:=e^{-y}$, we have
$$\frac{1}{q}\frac{dq}{dx}=-\frac{dy}{dx}=\frac{1}{x(1-x)z^2}$$ so
that \begin{equation}
\frac{dq}{dx}=\frac{q}{x(1-x)z^2}.\label{dq}
\end{equation} From \eqref{alt-dz}, \eqref{mpxz}, and \eqref{dq}, we obtain
\begin{align*}
\frac{d\mP(q)}{dq}&=\frac{\frac{d\mP(q)}{dx}}{\frac{dq}{dx}}\\
                  &=\frac{\frac{d}{dx}\big(z^2(1-x)+4x(1-x)z\frac{dz}{dx}\big)}{\frac{q}{x(1-x)z^2}}\\
&=\frac{-z^2+(6-10x)z\frac{dz}{dx}+4x(1-x)(\frac{dz}{dx})^2+4x(1-x)z\frac{d^2z}{dx^2}}{\frac{q}{x(1-x)z^2}}\\
&=\frac{(2-2x)z\frac{dz}{dx}+4x(1-x)(\frac{dz}{dx})^2}{\frac{q}{x(1-x)z^2}},
\end{align*} so that
\begin{equation}\label{dmpxz}
q\frac{d\mP(q)}{dq}=2x(1-x)^2z^3\frac{dz}{dx}+4x^2(1-x)^2z^2\Big(\frac{dz}{dx}\Big)^2.
\end{equation}
Similarly, from \eqref{alt-dz}, \eqref{mpxz2}, and \eqref{dq}, we
obtain
\begin{equation}\label{dmpxz2}
q\frac{d\mP(q^2)}{dq}=-\frac{x(1-x)z^4}{2}+2x(1-x)^2z^3\frac{dz}{dx}+2x^2(1-x)^2z^2\Big(\frac{dz}{dx}\Big)^2.
\end{equation}
In a similar manner, we find that
\begin{align}
q\frac{d\mE(q)}{dq}&=x(1-x)z^4+2x(1-x)(1+x)z^3\frac{dz}{dx},\label{dmexz}\\
q\frac{d\mE(q^2)}{dq}&=-\frac{x(1-x)z^4}{2}+x(1-x)(2-x)z^3\frac{dz}{dx}.\label{dmexz2}
\end{align}

Next, using Lemma \ref{BCB} and using \eqref{mpxz}, \eqref{mexz},\eqref{mqxz}, \eqref{mpxz2}, \eqref{mexz2}, \eqref{mqxz2}, \eqref{mpxz-}, \eqref{mexz-}, and \eqref{mqxz-}, we obtain the following identities.
\begin{theorem}\label{listofidentity}
Recall that $\mP$, $\mE$, and $\mQ$ are defined by \eqref{mp}, \eqref{me}, and \eqref{mq}, respectively, and that $\varphi(q)$ and $\psi(q)$ are defined in \eqref{phi} and \eqref{psi}, respectively. Then 
{\allowdisplaybreaks\begin{align}
\mQ(q)&=\varphi^8(-q),\label{mq-v} \\
16\psi^4(q^2)+\varphi^4(q)&=\mE(q),\label{pv-e} \\
2\mE(q^2)+\mE(q)&=3\varphi^4(q),\label{2ee-3v}\\
\varphi^4(q)\mE(q)+\mQ(q^2)&=2\varphi^8(q),\label{vemq-2v}\\
\mE(q)-\mE(q^2)&=24q\psi^4(q^2), \label{ee-p}\\
\mP(q)-\mP(-q)&=16q\psi^4(q^2),\label{mpmp-p}\\
\mQ(q)+\mQ(-q)&=32q(8\psi^8(q^2)-\psi^8(q)),\label{mqmq-p}\\
\mE^2(q)-\mQ(q)&=64q\psi^8(q).\label{emq-p}
%\mQ(q^2)&=(2\mP(q^2)-\mP(q))\varphi^4(q),\label{mqP}\\
%\mQ(q)&=(2\mP(q^2)-\mP(q))^2,\label{QPP}\\
%\mE(q)&=\frac{2\mQ(q^2)-\mQ(q)}{2\mP(q^2)-\mP(q)}.\label{ef}
\end{align}}
\end{theorem}

\pf \eqref{mq-v}. The result is clear from \eqref{phixz-} and \eqref{mqxz}.
\medskip
\pf \eqref{pv-e}. The equality $$16\psi^4(q^2)+\varphi^4(q)=xz^2+z^2=(1+x)z^2=\mE(q)$$ follows from \eqref{phixz}, \eqref{psixz2} and \eqref{mexz}.
\medskip
\pf \eqref{2ee-3v}. Employing \eqref{mexz} and \eqref{mexz2}, we
have $$2\mE(q^2)+\mE(q)=3z^2.$$ So the proof is completed by using \eqref{phixz}.
\medskip
\pf \eqref{vemq-2v}. By \eqref{phixz}, \eqref{mexz}, and \eqref{mqxz}, we find that 
$$\varphi^4(q)\mE(q)+\mQ(q^2)=z^4(1+x)+z^4(1-x)=2z^4=2\varphi^8(q).$$

\pf \eqref{ee-p}. By using \eqref{mexz}, \eqref{mexz2}, and \eqref{psixz2} we obtain
$$\mE(q)-\mE(q^2)=\frac{3}{2}xz^2=24q\psi^4(q^2).$$

\pf \eqref{mpmp-p}. From \eqref{mpxz} and \eqref{mpxz-}, we find
that $$\mP(q)-\mP(-q)=(1-x)z^2-(1-2x)z^2=xz^2=16q\psi^4(q^2).$$

\pf \eqref{mqmq-p}. By the definition of $\mQ$, we obtain
{\allowdisplaybreaks\begin{align*}
\mQ(q)+\mQ(-q)=&-16\sum_{n=1}^{\infty} (2n-1)^3q^{2n-1} \Big( \frac{1}{1-q^{2n-1}}+\frac{1}{1+q^{2n-1}} \Big)\\
=&-32\sum_{n=1}^{\infty}\frac{(2n-1)^3q^{2n-1}}{1-q^{4n-2}}=32q(8\psi^8(q^2)-\psi^8(q)),
\end{align*}}
where we use Example(ii) in \cite[p. 139]{BCB3}.
\medskip
\pf \eqref{emq-p}. From \eqref{mexz} and \eqref{mqxz}, we see that
{\allowdisplaybreaks\begin{align*}
\mE^2(q)-\mQ(q)=&4xz^4=\left(\frac{1}{16}xz^2\right)\left(64z^2\right)\\
=&(q\psi^4(q^2))\cdot (64\varphi^4(q)),
\end{align*}}
where the last equality follows from \eqref{phixz} and
\eqref{psixz}. After employing the fact~\cite[Entry 25,
p. 40]{BCB3} $$\varphi(q) \psi(q^2)=\psi^2(q),$$ we achieve the
desired result.

%%%%%%%%%%%%%%%%%%%%%%%%%%%%%%%%%%%%%%%%%%%%%%%%%%%%%%%%%%%%%%%%%%%%%%%%%%%%%
\section{Representations of certain infinite series}

In this section, we derive some representations of the infinite series connected with the functions $\mP(q)$, $\mE(q)$, and $\mQ(q)$.

\begin{theorem}
We have
\begin{equation}
1-24\sum_{n=1}^{\infty}\frac{2n-1}{e^{(2n-1)y}+1}=(1-2x)z^2. \label{2e-e}
\end{equation}
\end{theorem}
\begin{proof}
 From \eqref{mexz} and \eqref{mexz2}, we find that
\begin{equation}
2\mE(q^2)-\mE(q)=2(1-\frac{x}{2})z^2-(1+x)z^2=(1-2x)z^2.
\end{equation}
On the other hand, by the definition of $\mE(q)$ in \eqref{me}, we know
that {\allowdisplaybreaks\begin{align*}
2\mE(q^2)-\mE(q)&=2\Big(1+24\sum_{n=1}^{\infty}\frac{n}{e^{2ny}+1}\Big)
-\Big(1+24\sum_{n=1}^{\infty}\frac{n}{e^{ny}+1}\Big)\\
&=1+24\sum_{n=1}^{\infty}\frac{2n}{e^{2ny}+1}-24\sum_{n=1}^{\infty}
\Big(\frac{2n}{e^{2ny}+1}+\frac{2n-1}{e^{(2n-1)y}+1}\Big)\\
&=1-24\sum_{n=1}^{\infty}\frac{2n-1}{e^{(2n-1)y}+1}.
\end{align*}}
\end{proof}
\noindent \textbf{Remark.} We can compare this result with some of the results in \cite[Entry 13, p. 127]{BCB3}. For example \cite[(viii)]{BCB3}, we have
$$1+24\sum_{n=1}^{\infty}\frac{n}{e^{ny}+1}=(1+x)z^2.$$ By using the representations for $P(q)$, $Q(q)$ and $R(q)$ and their algebraic relations, Berndt \cite{BCB3} also lists further representations, such as
{\allowdisplaybreaks\begin{align}
1+8\sum_{n=1}^{\infty}\frac{(-1)^{n-1}n^5}{e^{ny}-1}&=(1-x)(1-x^2)z^6,\label{ss5}\\
17-32\sum_{n=1}^{\infty}\frac{(-1)^{n-1}n^7}{e^{ny}-1}&=(1-x)^2(17-2x+17x^2)z^8,\label{ss7}\\
1+8\sum_{n=1}^{\infty}\frac{(-1)^nn^5}{e^{2ny}-1}&=(1-x)(1-\frac{1}{2}x)z^6,\label{2ss5}\\
17-32\sum_{n=1}^{\infty}\frac{(-1)^nn^7}{e^{2ny}-1}&=(1-x)(17-17x+2x^2)z^8.\label{2ss7}
\end{align}}

\begin{theorem}\label{sinh}
We have 
{\allowdisplaybreaks\begin{align}
\sum_{n=1}^{\infty}\frac{(-1)^{n-1}n^3}{\sinh(ny)}=&\frac{1}{8}x(1-x)z^4, \label{sinh3}\\
\sum_{n=1}^{\infty}\frac{(-1)^{n-1}n^5}{\sinh(ny)}=&\frac{1}{8}x(1-x)(1-2x)z^6, \label{sinh5}\\
\sum_{n=1}^{\infty}\frac{(-1)^{n-1}n^7}{\sinh(ny)}=&\frac{1}{16}x(1-x)(2-17x+17x^2)z^8.\label{sinh7}
\end{align}}
\end{theorem}

\begin{proof}
We use the elementary fact
\begin{equation}
\frac{1}{x-1}-\frac{1}{x^2-1}=\frac{x}{x^2-1}=\frac{1}{x-x^{-1}}.\label{partialfrac}
\end{equation}
To prove \eqref{sinh3}, we simply use the definition of $\mQ$ and
\eqref{partialfrac} to obtain
$$\sum_{n=1}^{\infty}\frac{(-1)^{n-1}n^3}{e^{ny}-e^{-ny}}=-\frac{1}{16}\left\{\mQ(e^{-y})-\mQ(e^{-2y})\right\}
=\frac{1}{16}x(1-x)z^4,$$ where we used \eqref{mqxz} and
\eqref{mqxz2} in the last equality. For \eqref{sinh5}, by
\eqref{partialfrac}, the sum to be evaluated is equal to
{\allowdisplaybreaks\begin{align*}
\sum_{n=1}^{\infty}\frac{(-1)^{n-1}n^5}{e^{ny}-e^{-ny}}&=\frac{1}{8}
\left\{\Big(1+8\sum_{n=1}^{\infty}\frac{(-1)^nn^5}{e^{ny}-1}\Big)-
\Big(1+8\sum_{n=1}^{\infty}\frac{(-1)^nn^5}{e^{2ny}-1}\Big)\right\}\\
&=\frac{1}{8}\{(1-x)(1-x^2)z^6-(1-x)(1-\frac{1}{2}x)z^6\}\\
&=\frac{1}{8}x(1-x)(1-2x)z^6,
\end{align*}} where we employ \eqref{ss5} and \eqref{2ss5} to derive \eqref{sinh5}. In a similar manner, we can deduce \eqref{sinh7} by using \eqref{partialfrac} and applying \eqref{ss7},
\eqref{2ss7}.
\end{proof}

Applying the {\it duplication} process to \eqref{sinh3}--\eqref{sinh7}, respectively, gives
{\allowdisplaybreaks\begin{align}
\sum_{n=1}^{\infty}\frac{(-1)^{n-1}n^3}{\sinh(2ny)}=&\frac{1}{32}\sqrt{1-x}(1-\sqrt{1-x})^2z^4, \label{2sinh3}\\
\sum_{n=1}^{\infty}\frac{(-1)^{n-1}n^5}{\sinh(2ny)}=&\frac{1}{64}\sqrt{1-x}(1-\sqrt{1-x})^2(x-2+6\sqrt{1-x})z^6,\label{2sinh5}\\
\sum_{n=1}^{\infty}\frac{(-1)^{n-1}n^7}{\sinh(2ny)}=&\frac{1}{512}(1-x)(1-\sqrt{1-x})^2(76\sqrt{1-x}-30(2-x)+x^2)z^8.\label{2sinh7}
\end{align}}
\noindent \textbf{Remark.} We can compare the above results with some results in \cite[Entry 15, p. 132]{BCB3}. For example, Berndt proved that $$\sum_{n=1}^{\infty}\frac{n^3}{\sinh(ny)}=\frac{1}{8}xz^4.$$

%%%%%%%%%%%%%%%%%%%%%%%%%%%%%%%%%%%%%%%%%%%%%%%%%%%%%%%%%%%%%%%%%%%%%%%%%%%%%%
\section{Some convolution sums of $\wt_s(n)$ and $\wh_s(n)$}

We begin this section by recalling again the three differential equations satisfied by $\mP(q)$, $\mE(q)$, and $\mQ(q)$:
{\allowdisplaybreaks\begin{align}
q\frac{d\mP(q)}{dq}&=\frac{\mP^2(q)-\mQ(q)}{4}, \label{l-dmp}\\
q\frac{d\mE(q)}{dq}&=\frac{\mE(q)\mP(q)-\mQ(q)}{2}, \label{l-dme}\\
q\frac{d\mQ(q)}{dq}&=\mP(q)\mQ(q)-\mE(q)\mQ(q). \label{l-dmq}
\end{align}}
It is then easy to show that the following convolution sums
follow from \eqref{l-dmp}--\eqref{l-dmq}.

\begin{theorem}\label{basic-convolution}
{\allowdisplaybreaks\begin{align}
4&\sum_{m<n}\wt(m)\wt(n-m)=-\wt_3(n)+(2n-1)\wt(n),\label{t3}\\
24&\sum_{m<n}\wh(m)\wt(n-m)=-2\wt_3(n)+(6n-3)\wh(n)-\wt(n),\label{ht}\\
16&\sum_{m<n}(\wt(m)-3\wh(m))\wt_3(n-m)=2n\wt_3(n)+\wt(n)-3\wt(n).\label{tt3-ht3}
\end{align}}
\end{theorem}
\begin{proof}
We can rewrite \eqref{l-dmp} as $$\mP^2(q)=\mQ(q)+4q\frac{d\mP(q)}{dq}.$$
Then we have
\begin{equation}\label{eq1}
\left(1+8\sum_{n=1}^{\infty}
\wt(n)q^n\right)^2=\left(1-16\sum_{n=1}^{\infty}
\wt_3(n)q^n\right)+32\sum_{n=1}^{\infty} n\wt(n)q^n.
\end{equation}
Equating the coefficients of $q^n$ on both sides of \eqref{eq1},
we obtain \eqref{t3}. In a similar manner, the remaining two
convolution sums \eqref{ht} and \eqref{tt3-ht3} can be derived
from \eqref{l-dme} and \eqref{l-dmq}, respectively.
\end{proof}
It naturally arises to question the evaluation of the sum $$\sum_{m<n}\wh(m)\wh(n-m),$$ which will be mentioned in the following theorem.
\begin{theorem}
\begin{equation}\label{whwh}
36\sum_{m<n}\wh(m)\wh(n-m)=\left\{\begin{array}{ll}
                 -3\wh(n)+3\wt_3(n), & \text{if n is odd},\\
           -3\wh(n)-5\wt_3(n)+4\wt_3(n/2), & \text{if n is even}.
                       \end{array} \right.
\end{equation}
\end{theorem}
\begin{proof}
By using \eqref{mexz}, \eqref{mqxz}, \eqref{mqxz2} and \eqref{sq1xz}, we can easily derive the identity
{\allowdisplaybreaks\begin{align*}
\mE^2(q) &= z^4(1+x)^2\\&=z^4(5(1-x)^2-4(1-x)+4x(2-x))\\&=5\mQ(q)-4\mQ(q^2)+128\mQ_{1,2}(q).
\end{align*}} Equating coefficients of $q^n$ gives the desired evaluation.
\end{proof}
\noindent\textbf{Remark.} We point out that certain of the convolution sums considered here can be evaluated from known results in an elementary manner. For example, by using the relation \eqref{wh-s}, we have that
{\allowdisplaybreaks\begin{align*}
\sum_{m<n}\wh(m)\wh(n-m)=&\sum_{m<n}(\sigma(m)-2\sigma(m/2))(\sigma(n-m)-2\sigma((n-m)/2))\\=&\sum_{m<n}\sigma(m)\sigma(n-m)-2\sum_{m<n}\sigma(m/2)\sigma(n-m)\\&-2\sum_{m<n}\sigma((n-m)/2)\sigma(m)+4\sum_{m<n}\sigma(m/2)\sigma((n-m)/2)\\
=&A(n)-4B(n)+4A(n/2),
\end{align*}}
where $$A(n)=\sum_{m<n}\sigma(m)\sigma(n-m)$$ and $$B(n)=\sum_{m<n/2}\sigma(m)\sigma(n-2m).$$ The values of $A(n)$ and $B(n)$ are given in \cite{Huard}.

\begin{theorem}
\begin{equation}
16\sum_{m<n}\wt(m)\wt_3(n-m)=-\wt_5(n)+2(n-1)\wt_3(n)+\wt(n).\label{t5}
\end{equation}
\end{theorem}
\begin{proof}
 From the differential equation \eqref{l-dmq}, we find that
\begin{equation}\label{e-t5}
1+8\sum_{n=1}^{\infty}\wt_5(n)q^n=\mE(q)\mQ(q)=\mP(q)\mQ(q)-q\frac{d\mQ(q)}{dq},
\end{equation} where the second equality comes from \cite[(2.2.8)]{Hahn}. So we complete the proof by equating the coefficients of $q^n$ on both sides of \eqref{e-t5}.
\end{proof}

\noindent \textbf{Remark.} Note that the identities \eqref{t3} and \eqref{t5} are analogues of the identities \eqref{sigma1} and \eqref{sigma2}, respectively,
which we mentioned in Section 1. The identity \eqref{t3} was also proved by Glaisher~\cite{Glaisher-sum} by theory of the elliptic functions.

Using the formulas given in Section 2, for $r \neq s$ and $r,s\in
\{1,2\}$, we determine the products $\mP(q^r)\mP(q^s)$ and
$\mP(q^r)\mE(q^s)$ as linear combinations of $\mQ(q)$, $\mQ(q^2)$
and the derivatives of $\mP(q)$, $\mP(q^2)$, $\mE(q)$, and $\mE(q^2)$.
\begin{theorem}\label{identity1}
We have {\allowdisplaybreaks\begin{align}
\mP(q)\mP(q^2)&=\mQ(q^2)+q\frac{d\mP(q)}{dq}+2q\frac{d\mP(q^2)}{dq},\label{mpmp2}\\
\mP(q^2)\mE(q)&=\mQ(q^2)+\frac{1}{3}\Big(q\frac{d\mE(q)}{dq}+2q\frac{d\mE(q^2)}{dq}\Big)+\Big(q\frac{d\mP(q)}{dq}-2q\frac{\mP(q^2)}{dq}\Big),\label{mp2me}\\
\mP(q)\mE(q^2)&=\frac{1}{2}(3\mQ(q^2)-\mQ(q))+2q\frac{d\mE(q^2)}{dq}.\label{mpme}
\end{align}}
\end{theorem} 
We just give the proof of \eqref{mpmp2}, since the remaining
 proofs are similar.
\medskip
\pf \eqref{mpmp2}. By \eqref{mqxz2}, \eqref{dmpxz}, and \eqref{dmpxz2}, we have 
{\allowdisplaybreaks\begin{align*}
\mQ(q^2)+&q\frac{d\mP(q)}{dq}+2q\frac{d\mP(q^2)}{dq}\\
=&(1-x)z^4+2x(1-x)^2z^3\frac{dz}{dx}+4x^2(1-x)^2z^2\Big(\frac{dz}{dx}\Big)^2\\
&-x(1-x)z^4+4x(1-x)^2z^3\frac{dz}{dx}+4x^2(1-x)^2z^2\Big(\frac{dz}{dx}\Big)^2\\
=&(1-x)^2z^4+6x(1-x)^2z^3\frac{dz}{dx}+8x^2(1-x)^2z^2\Big(\frac{dz}{dx}\Big)^2\\
=&\mP(q)\mP(q^2),
\end{align*}} where we simply calculate the product of \eqref{mpxz} and \eqref{mpxz2}. This completes the proof of \eqref{mpmp2}. The remaining formulas can be proved similarly.

Equating the coefficients of $q^n$ on both sides in the three
identities in Theorem~\ref{identity1}, we obtain the next theorem.

\begin{theorem}\label{convolution1}
We have {\allowdisplaybreaks\begin{align}
8\sum_{m<n/2}\wt(m)\wt(n-2m)=&-\wt_3(n/2)+(n-1)\wt(n)+(2n-1)\wt(n/2),\label{tt2}\\
24\sum_{m<n/2}\wt(m)\wh(n-2m)=&-2\wt_3(n/2)+(2n-3)\wh(n)+4n\wh(n/2)\label{th2}\\
                             &+n\wt(n)-(2n+1)\wt(n/2),\nonumber\\
24\sum_{m<n/2}\wh(m)\wt(n-2m)=&\wt_3(n)-3\wt_3(n/2)+(6n-3)\wh(n/2)-\wt(n).\label{ht2}
\end{align}}
\end{theorem}

The next theorem shows that for $r \in \{0,1\}$ and $s \in
\{1,2\}$, the products of the form $\mP_{r,2}(q)\big(-1+\mP(q^s)\big)$ and
$\mE_{r,2}(q)\big(-1+\mP(q^s)\big)$ can be expressed as linear combinations of $\mP(q)$, $\mP(q^2)$, $\mE(q)$, $\mE(q^2)$, $\mQ(q)$, $\mQ(q^2)$, and the derivatives of $\mP(q)$, $\mP(q^2)$, $\mE(q)$, and $\mE(q^2)$. A MAPLE program was run to determine the identities.
\begin{theorem}\label{identity2}
We have {\allowdisplaybreaks\begin{align}
\mP_{0,2}(q)\big(-1+\mP(q)\big)
=&\frac{1}{8}+\frac{1}{16}(\mQ(q)+\mQ(q^2))-\frac{1}{24}(\mE(q)-7\mE(q^2))\label{sp0mp}\\
&-\frac{1}{2}\mP(q^2)+\frac{1}{2}q\frac{d\mP(q)}{dq}
-\frac{1}{12}\Big(q\frac{\mE(q)}{dq}+q\frac{\mE(q^2)}{dq}\Big),\nonumber\\
\mP_{0,2}(q)\big(-1+\mP(q^2)\big)
=&\frac{1}{8}+\frac{1}{8}(\mQ(q^2)-3\mP(q^2)+\mE(q^2))\label{sp0mp2}\\
&+\frac{1}{16}\Big(q\frac{\mP(q)}{dq}+6q\frac{\mP(q^2)}{dq}\Big)
-\frac{1}{48}\Big(q\frac{\mE(q)}{dq}+2q\frac{\mE(q^2)}{dq}\Big),\nonumber\\
\mP_{1,2}(q)\big(-1+\mP(q)\big)
=&\frac{1}{16}(\mQ(q)-\mQ(q^2))-\frac{1}{24}(\mE(q)-\mE(q^2))\label{sp1mp}\\
&+\frac{1}{12}\Big(q\frac{\mE(q)}{dq}-q\frac{\mE(q^2)}{dq}\Big),\nonumber\\
\mP_{1,2}(q)\big(-1+\mP(q^2)\big)
=&\frac{1}{24}(\mE(q^2)-\mE(q))+\frac{1}{24}
\Big(q\frac{\mE(q)}{dq}-q\frac{\mE(q^2)}{dq}\Big),\label{sp1mp2}\\
\mE_{0,2}(q)\big(-1+\mP(q)\big)
=&\frac{1}{24}+\frac{1}{16}(\mQ(q^2)-3\mQ(q))-\frac{1}{24}\mP(q)\label{se0mp}\\
&-\frac{1}{24}\mE(q^2)-\frac{1}{12}q\frac{\mE(q^2)}{dq},\nonumber\\
\mE_{0,2}(q)\big(-1+\mP(q^2)\big)
=&\frac{1}{24}+\frac{1}{24}(\mQ(q^2)-\mP(q^2)-\mE(q^2))
+\frac{1}{24}q\frac{\mE(q^2)}{dq}.\label{se0mp2}
\end{align}}
\end{theorem}
Again we just give the proof of \eqref{sp0mp}, since the remaining
 proofs are similar.
\medskip
\pf \eqref{sp0mp}. By \eqref{mqxz} and \eqref{mqxz2}, we have
$$\mQ(q)+\mQ(q^2)=2z^4-3xz^4+x^2z^4,$$ and from \eqref{mexz} and
\eqref{mexz2}, $$\mE(q)-7\mE(q^2)=-6z^2+\frac{9}{2}xz^2,$$ and from
\eqref{dmexz} and \eqref{dmexz2} 
$$q\frac{\mE(q)}{dq}+q\frac{\mE(q^2)}{dq}=\frac{1}{2}xz^4-\frac{1}{2}x^2z^4+4x(1-x)z^3\frac{dz}{dx}+x^2(1-x)z^3\frac{dz}{dx}.$$ Therefore by \eqref{mpxz2}, \eqref{dmpxz}, and the previous three equalities, we finally
obtain {\allowdisplaybreaks\begin{align*}
\frac{1}{8}+&\frac{1}{16}(\mQ(q)+\mQ(q^2))-\frac{1}{24}(\mE(q)-7\mE(q^2))-\frac{1}{2}\mP(q^2)\\
            &+\frac{1}{2}q\frac{d\mP(q)}{dq}
            -\frac{1}{12}\Big(q\frac{\mE(q)}{dq}+q\frac{\mE(q^2)}{dq}\Big)\\
           &=\frac{1}{8}+\frac{1}{16}(2z^4-3xz^4+x^2z^4)-\frac{1}{24}(-6z^2+\frac{9}{2}xz^2)\\
           &-\frac{1}{2}\left(z^2(1-x)+2x(1-x)z\frac{dz}{dx}\right)\\
           &+\frac{1}{2}\left(2x(1-x)^2z^3\frac{dz}{dx}+4x^2(1-x)^2z^2\Big(\frac{dz}{dx}\Big)^2\right)\\
&-\frac{1}{12}\left(\frac{1}{2}xz^4-\frac{1}{2}x^2z^4+4x(1-x)z^3\frac{dz}{dx}+x^2(1-x)z^3\frac{dz}{dx}\right)\\
=&\frac{1}{8}-\frac{1}{4}z^2+\frac{5}{16}xz^2-x(1-x)z\frac{dz}{dx}
+\frac{1}{8}z^4-\frac{5}{16}xz^4+x(1-x)z^3\frac{dz}{dx}\\
&+\frac{3}{16}x^2z^4-\frac{5}{4}x^2(1-x)z^3\frac{dz}{dx}+2\left(x(1-x)z\frac{dz}{dx}\right)^2\\
=&\mP_{0,2}(q)\big(-1+\mP(q)\big).
\end{align*}}

Equating the coefficients of $q^n$ on both sides of the six formulas in Theorem \ref{identity2}, we obtain the following convolution sums.
\begin{theorem}\label{convolution2}
We have {\allowdisplaybreaks\begin{align}
&8\sum_{m<n/2}\wt(2m)\wt(n-2m)=-\wt_3(n)-\wt_3(n/2)+4n\wt(n)-4\wt(n/2)\label{t2t2}\\
                              &\hspace{2in}-(2n+1)\wh(n)+(2n+7)\wh(n/2),\nonumber\\
&8\sum_{m<n/2}\wt(2m)\wt(n/2-m)=-2\wt_3(n/2)+n/2\wt(n)+(3n-3)\wt(n/2)\label{t2t}\\
                              &\hspace{2in}-n/2\wh(n)-(n-3)\wh(n/2),\nonumber\\
&8\sum_{m<(n+1)/2}\wt(2m-1)\wt(n-(2m-1))=-\wt_3(n)+\wt_3(n/2)\label{t1t1}\\
                          &\hspace{2in}+(2n-1)\wh(n)-(2n-1)\wh(n/2),\nonumber\\
&8\sum_{m<(n+1)/2}\wt(2m-1)\wt((n+1)/2-m)=(n-1)\wh(n)-(n-1)\wh(n/2),\label{t1t}\\
&8\sum_{m<n/2}\wh(2m)\wt(n-2m)=\frac{1}{3}\wt_3(n)-\wt_3(n/2)+(2n-1)\wh(n/2),\label{h2h2}\\
&8\sum_{m<n/2}\wh(2m)\wt(n/2-m)=-\frac{2}{3}\wt_3(n/2)-\frac{1}{3}\wt(n/2)+(n-1)\wh(n/2).\label{h2h}
\end{align}}
\end{theorem}

%%%%%%%%%%%%%%%%%%%%%%%%%%%%%%%%%%%%%%%%%%%%%%%%%%%%%%%%%%%%%%%%%%%%%%%%%%%%%%%%%%%%%%%%%%%%%%%%%%%%%%%%%%%%%%%

\section{On the representations of integers as sums of squares and triangular numbers}

It is immediate from the definitions of $\varphi(q)$ and $\psi(q)$ in \eqref{phi} and \eqref{psi}, respectively, that if
\begin{equation}\label{ssum}
\varphi^s(q):=\sum_{n=0}^{\infty} r_s(n)q^n \end{equation} and
\begin{equation}\label{stri}
\psi^s(q):=\sum_{n=0}^{\infty}\delta_s(n)q^n, \end{equation} then
$r_s(n)$ and $\delta_s(n)$ are the number of representations of
$n$ as a sum of $s$ squares and $s$ triangular numbers,
respectively. Clearly $r_s(0)=\delta_s(0)=1$. Here, for each non negative integer $n$, the triangular number $T_n$ is defined by $$T_n:=\frac{n(n+1)}{2}.$$
By using the representations and identities derived in Section 2, we find expressions for $r_s(n)$ and $\delta_s(n)$, $s=4, 8$, as sums of our functions $\wt(n)$, $\wh(n)$, and $\wt_3(n)$.

\begin{theorem}\label{convolution3}
For each positive integer $n$, we have
{\allowdisplaybreaks\begin{align}
r_4(n)=&16\wh(n/2)+8\wh(n),\label{4sum}\\
\delta_4(n)=&\wt(2n+1),\label{4tri}\\
r_8(n)=&16(-1)^{n-1}\wt_3(n)\label{8sum}\\
8\delta_8(n)=&\wt_3(n+1)-\wt_3(2(n+1)).\label{8tri}
\end{align}}
\end{theorem}

\pf \eqref{4sum}. The identity \eqref{2ee-3v} is equivalent to the identity
\begin{equation}\label{i-4sum}
3\sum_{n=1}^{\infty}r_4(n)q^n=48\sum_{n=1}^{\infty}\wh(n)q^{2n}+24\sum_{n=1}^{\infty}\wh(n)q^n.
\end{equation}
The identity \eqref{4sum} follows after equating the coefficients
of $q^n$ on both sides of \eqref{i-4sum}.

\medskip
\pf \eqref{4tri}. By \eqref{sp1} and \eqref{mpmp-p}, we have
\begin{equation}
q\psi^4(q^2)=\mP_{1,2}(q). \label{i-4tri}
\end{equation} Hence we have
$$q\sum_{n=0}^{\infty}\delta_4(n)q^{2n}=\sum_{n=0}^{\infty}\wt(2n+1)q^{2n+1},$$
which is the identity \eqref{4tri}.

\medskip
\pf \eqref{8sum}. It is clear from \eqref{mqxz-} that
\begin{equation}\label{i-8sum}
\sum_{n=1}^{\infty}r_8(n)q^n=-16\sum_{n=1}^{\infty}\wt_3(n)(-q)^n.
\end{equation}

\medskip
\pf \eqref{8tri}. From \eqref{psixz2}, \eqref{mqxz2}, and \eqref{sq0xz}, we have $$8q^2\psi^8(q^2)=\frac{1}{16}-\frac{1}{16}\mQ(q^2)-\mQ_{0,2}(q).$$ Hence we derive
\begin{equation}\label{i-8tri}
8\sum_{n=1}^{\infty}\delta_8(n-1)q^{2n}=\sum_{n=1}^{\infty}\wt_3(n)q^{2n}-\sum_{n=1}^{\infty}\wt_3(2n)q^{2n}.
\end{equation}
 Equating the coefficients of $q^n$ on both sides of \eqref{i-8tri}, we obtain the desired result.

\medskip
\noindent \textbf{Remarks.} 
Jacobi \cite{Jacobi-note, Jacobi-funda, Jacobi-de} showed that $r_4(n)$ is $8$ times the sum of the divisors of $n$ that are not multiples of $4,$ that is,
\begin{equation}
r_4(n)=8(\sigma(n)-4\sigma(n/4)). \label{r4}
\end{equation}
Many proofs of \eqref{r4} have been given; see for example \cite{Andrews}, \cite[p. 15]{Berndt}. Spearman and Williams \cite{Will4} gave the simplest arithmetic proof of this formula. If we use $\wh(n)=\sigma(n)-2\sigma(n/2)$ from \eqref{wh-s}, then we note that our expression for $r_4(n)$ in \eqref{4sum} is the same as \eqref{r4}. By the fact \eqref{odd}, we can express \eqref{4tri} as
\begin{equation}\label{t4}
\delta_4(n)=\sigma(2n+1).
\end{equation} The formula \eqref{t4} is proved in an elemantary way \cite[Theorem 10]{Huard}, and in using modular forms \cite[Theorem 3]{Ono-tri}. The evaluation of $\delta_4(n)$ goes back to Legendre \cite{Dickson}, \cite{Legendre}. The formula \eqref{8sum} first appeared implicitly in the work of Jacobi \cite{Jacobi-funda}, and explicitly in the work of Eisenstein \cite{Eisenstein}. Williams \cite{Will8} gave an arithmetic proof of this formula by showing that $$r_8(n)=16\sigma_3(n)-32\sigma_3(n/2)+256\sigma_3(n/4).$$ Using the theory of modular forms, Ono, Robins, and Wahl \cite[Theorem 5]{Ono-tri} derive a formula for $\delta_8(n)$, namely
\begin{equation}
\delta_8(n)=\sigma_3(n+1)-\sigma_3((n+1)/2).\label{Ono8}
\end{equation} Formula \eqref{Ono8} is also proved in an elementary way in \cite[Theorem 12]{Huard}. It is not hard to show that \eqref{Ono8} is the same expression as \eqref{8tri}. From \eqref{wt-s}, we deduce that
\begin{equation}
\wt_3(n)=\sigma_3(n)-16\sigma_3(n/2). \label{8tri-s}
\end{equation} Then
\begin{align}
8\delta_8(n)=&\wt_3(n+1)-\wt_3(2(n+1))\nonumber\\
=&\sigma_3(n+1)-16\sigma_3((n+1)/2)\nonumber\\
&-(\sigma_3(2(n+1))-16\sigma_3(n+1))\nonumber\\
=&8(\sigma_3(n+1)-\sigma_3((n+1)/2)), \label{Ono}
\end{align} where in the last equality, we use the identity
\begin{equation}
\sigma_3(2n)=9\sigma_3(n)-8\sigma_3(n/2). \label{will}
\end{equation}
The identity \eqref{will} can be proved by letting $n:=2^aN,$ $N$ is odd, and then by considering the cases, $a=0$ and $a>0$. After dividing both sides of \eqref{Ono} by $8$, we have the desired identity \eqref{Ono8}.

%%%%%%%%%%%%%%%%%%%%%%%%%%%%%%%%%%%%%%%%%%%%%%%%%%%%%%%%%%%%%%%%%%%%%%%%%%%%%%%%%%%%%%%%%%%%%%%%%%%%%%%%%%%%%%%%%%%%%%
\section{Some partition congruences}

If $r$ is a non--zero integer, we define the function $p_r(n)$ by
\begin{equation}\label{p_r}
\sum_{n=0}^{\infty}p_r(n)q^n:=\prod_{n=1}^{\infty}(1-q^n)^r.
\end{equation}
Note that $p_{-1}(n)=p(n)$, the ordinary partition function. A positive
integer $n$ has $k$ {\it colors} if there are $k$ copies of $n$
available and all of them are viewed as distinct objects.
Partitions of positive integers into parts with colors are called
{\it colored partitions}. For example, if $1$ is allowed to have
$2$ colors, say $r(red)$, and $g(green)$, then all colored
partitions of $2$ are $2, 1_r+1_r, 1_g+1_g, 1_r+1_g$. Letting $p_{e,r}(n)$ and $p_{o,r}(n)$ denote the number of
$r$--colored partitions into an even (respectively, odd) number of
distinct parts, it is easy to see that
\begin{equation}\label{pe-po}
p_r(n)=p_{e,r}(n)-p_{o,r}(n),
\end{equation}
when $r$ is a positive integer.

We prove a congruence for the function $\mu (n)$ which is defined
by \begin{equation}\label{mu}
\sum_{n=0}^{\infty} \mu (n) q^n:=\prod_{n=1}^{\infty}
(1-q^n)^8(1-q^{2n})^8.
\end{equation}It follows that
\begin{equation}\label{mue-muo}
\mu (n)=\mu_e(n)-\mu_o(n),
\end{equation}
where $\mu_e(n)$ and $\mu_o(n)$ are the number of $16$--colored
partitions into an even (respectively, odd) number of distinct
parts, where all the parts of the latter eight colors are even.

\begin{theorem}\label{color1}
If $\mu (n)$ is defined by \eqref{mue-muo}, $$\mu(3n-1) \equiv 0
\pmod 3.$$
\end{theorem}

We generally denote by $J$ an integral power series in $q$ whose
coefficients are integers.

\begin{proof}
It is obvious from \eqref{me} that $$\mE(q)=1+3J.$$ Also $n^3-n \equiv 0
\pmod 3$, and so, from \eqref{mp} and \eqref{mq}, we obtain
$$\mQ(q)=\mP(q)+3J.$$ Hence
{\allowdisplaybreaks\begin{align}
(\mE^2(q)-\mQ(q))\mQ(q)&=(\mE(q)(1+3J)-(\mP(q)+3J))\mQ(q)\nonumber\\
              &=\mE(q)\mQ(q)-\mP(q)\mQ(q)+3J. \label{eQPQ}
\end{align}}
By \eqref{mq-v} and \eqref{emq-p},  we find that
\begin{equation}\label{64}
(\mE^2(q)-\mQ(q))\mQ(q)=64q\psi^8(q)\varphi^8(-q)=64q\prod_{n=1}^{\infty}
(1-q^n)^8(1-q^{2n})^8,
\end{equation} where the last equality comes from the fact~\cite[p.~39]{BCB3}
$$\varphi(-q) \psi(q)=f(-q)f(-q^2),$$ where $f(-q)$ is defined by
\eqref{f}. On the other hand, observe that, from \eqref{mq} and \eqref{dmq},
\begin{equation}\label{nwt3}
16\sum_{n=1}^{\infty} n \wt_3(n)q^n=-q\frac{d\mQ(q)}{dq}=\mE(q)\mQ(q)-\mP(q)\mQ(q).
\end{equation} 
In summary, by \eqref{eQPQ}, \eqref{64}, and \eqref{nwt3}, we
conclude that
\begin{equation}\label{partition1}
64\sum_{n=0}^{\infty} \mu (n) q^{n+1}=16\sum_{n=1}^{\infty} n
\wt_3(n)q^n +3J.
\end{equation}
But the coefficient of $q^{3n}$ on the right side of
\eqref{partition1} is a multiple of $3$. So we obtain $$\mu(3n-1)
\equiv 0 \pmod 3.$$
\end{proof}

Secondly, we prove a congruence for the function $\nu (n)$ which
is defined by
\begin{equation}\label{nu}
\sum_{n=0}^{\infty} \nu (n)
q^n:=\prod_{n=1}^{\infty}(1-q^{2n})^8(1+q^n)^8.
\end{equation}
Thus $\nu(n)$ is the number of partitions of $n$ into $16$ colors,
$8$ appear at most once (say $S_1$), and $8$ are even and appear
at most once (say $S_2$), weighted by the parity of colors from the
set $S_2$.

\begin{theorem}
If $\nu (n)$ is defined by \eqref{nu}, then $$\nu(n-1) \equiv
\wt_3(n) \pmod 3 .$$
\end{theorem}

\begin{proof}
Recall from \eqref{emq-p} of Theorem \ref{listofidentity} that
\begin{equation}\label{partition-D}
\mE^2(q)-\mQ(q)=64q\prod_{n=1}^{\infty}\frac{(1-q^{2n})^8}{(1-q^{2n-1})^8}
=64q\prod_{n=1}^{\infty}(1-q^{2n})^8(1+q^n)^8, 
\end{equation} where, in the last equality, we used the fact \cite[(22.3)]{BCB3} $$\prod_{n=1}^{\infty}(1+q^n)=\prod_{n=1}^{\infty}(1-q^{2n-1})^{-1}.$$ Then, by
\eqref{nu} and \eqref{partition-D}, we deduce that
$$64\sum_{n=0}^{\infty} \nu (n)q^{n+1}=48\sum_{n=1}^{\infty}\wh(n)q^n+576\sum_{n=2}^{\infty}\sum_{m=1}^{n-1}\wh(m)\wh(n-m)q^n+16\sum_{n=1}^{\infty}\wt_3(n)q^n.$$
Comparing the coefficients of $q^n$ on both sides of the above equation, we obtain the identity
$$4\nu(n-1)=3\wh(n)+\wt_3(n)+36\sum_{m=1}^{n-1}\wh(m)\wh(n-m).$$
We then deduce that $$\nu(n-1) \equiv \wt_3(n) \pmod{3} .$$
\end{proof}

\medskip
\textbf{Acknowledgement.} I am deeply indebted to Professors B. C.
Berndt and K. S. Williams for their helpful comments and suggestions.

%%%%%%%%%%%%%%%%%%%%%%%%%%%%%%%%%%%%%%%%%%%%%%%%%%%%%%%%%%%%%%%%%%%%%%%%%%%%%%%%%%%%%%%%%%%%%%%%%%%%%%%%%


\begin{thebibliography}{99}
\bibitem{Glaisher-square}
J. W. L. Glaisher, \emph{On the square of the sereis in which the coefficients are the sums of the divisors of the exponents}, Mess. Math. \textbf{14} (1884), 156--163.
\bibitem{Glaisher-express}
J. W. L. Glaisher, \emph{Expressions for the first five powers of the series in which the coefficients are the sums of the divisors of the exponents}, Mess. Math. \textbf{15} (1885), 33--36.
\end{thebibliography}
\end{document}